\documentclass[11pt,draft]{amsart}
\usepackage{mathrsfs}
\usepackage{amsmath}
\usepackage{amssymb}
\usepackage{pifont}
\usepackage{amsfonts,enumerate}
\usepackage{graphics}
\usepackage{verbatim}
\usepackage{amsthm}

\numberwithin{equation}{section}
\newtheorem{theorem}{Theorem}[section]
\newtheorem{proposition}[theorem]{Proposition}
\newtheorem{lemma}[theorem]{Lemma}

\newtheorem{definition}[theorem]{Definition}
\newtheorem*{defi}{Definition}
\newtheorem{remark}[theorem]{Remark}

\newtheorem{pb}{Problem}

\newenvironment{rk}{\begin{remark}\rm}{\end{remark}}

\newcommand{\R}{\mathbb{R}} 
\newcommand{\N}{\mathbb{N}}

\newcommand{\E}{\mathcal{E}} 
\newcommand{\M}{\mathcal{M}} 
\renewcommand{\H}{\mathcal{H}} 
\renewcommand{\P}{\mathcal{P}} 
\newcommand{\BMO}{\mathcal{BMO}} 
  
\newcommand{\cond}{\mathrm{cond}}
\newcommand{\bmo}{\mathsf{bmo}}
\newcommand{\mo}{\mathsf{mo}} 
\newcommand{\h}{\mathsf{h}} 
\newcommand{\at}{\mathrm{at}}
\newcommand{\cqd}{\hfill$\Box$}

\begin{document}

\title{Atomic decomposition and interpolation for Hardy spaces of noncommutative martingales}


\author{Turdebek N. Bekjan}
\address{College of Mathematics and Systems Sciences, Xinjiang
University, Urumqi 830046, China}
\thanks{T. Bekjan is partially supported by NSFC grant
No.10761009.}

\author{Zeqian Chen}
\address{Wuhan Institute of Physics and Mathematics, Chinese Academy of
Sciences, P.O.Box 71010, 30 West Strict, Xiao-Hong-Shan, Wuhan
430071, China}
\thanks{Z. Chen is partially supported by NSFC grant
No.10775175.}

\author{Mathilde Perrin}
\address{Laboratoire de Math{\'e}matiques, Universit{\'e} de Franche-Comt{\'e}, \newline
25030 Besan\c{c}on Cedex, France}
\thanks{M. Perrin is partially supported by ANR 06-BLAN-0015.}

\author{Zhi Yin}
\address{Wuhan Institute of Physics and Mathematics, Chinese Academy of
Sciences, P.O.Box 71010, 30 West Strict, Xiao-Hong-Shan, Wuhan
430071, China and Graduate School, Chinese Academy of Sciences,
Wuhan 430071, China}

\date{}
\maketitle

\markboth{Turdebek N. Bekjan, Zeqian
Chen, Mathilde Perrin and Zhi Yin}%
{Atomic decomposition and interpolation for Hardy spaces of noncommutative martingales}

\begin{abstract}
We prove that atomic decomposition for the Hardy spaces $\h_1$ and $\H_1$ is valid for noncommutative martingales. 
We also establish that the conditioned Hardy spaces of noncommutative martingales $\h_p$ and $\bmo$ 
form interpolation scales with respect to both complex and real interpolations.
\end{abstract}

\section*{Introduction}

Atomic decomposition plays a fundamental role in the classical
martingale theory and harmonic analysis. 
For instance, atomic decomposition is a powerful tool for dealing with duality theorems, 
interpolation theorems and some fundamental inequalities both in martingale theory and harmonic analysis. 
Atoms for martingales are usually defined in terms of stopping times.
Unfortunately, the concept of stopping times is, up to now, not
well-defined in the generic noncommutative setting (there are some
works on this topic, see \cite{AC} and references therein). 
We note, however, that atoms can be defined without help of 
stopping times. Let us recall this in
classical martingale theory. 
Given a probability space $(\Omega,\mathcal{F}, \mu)$, 
let $(\mathcal{F}_n)_{n\geq 1}$ be an increasing filtration  of
$\sigma$-subalgebras of $\mathcal{F}$ such that $\mathcal{F} =
\sigma \big ( \cup_n \mathcal{F}_n \big )$ and let $(\mathcal{E}_n)_{n\geq 1}$ 
denote the corresponding family of conditional expectations. 
An $\mathcal{F}$-measurable function $a \in L_2$ is said to be an {\it atom} if there exist $n \in \mathbb{N}$ and $A \in
\mathcal{F}_n$ such that
\begin{enumerate}[{\rm (i)}]

\item  $\mathcal{E}_n (a) =0;$

\item  $\{ a \not= 0 \} \subset A;$

\item  $ \| a \|_2 \leq \mu (A)^{- 1/2}.$
\end{enumerate}
Such atoms are called {\it simple atoms} by Weisz \cite{W} 
and are extensively studied by him (see \cite{W1} and \cite{W}). 
Let us point out that atomic decomposition was first introduced in harmonic analysis by Coifman \cite{Coifman}. 
It is Herz \cite{H} who initiated atomic decomposition for martingale theory. 
Recall that we denote by $\H_1(\Omega)$ the space of martingales $f$ with respect to $(\mathcal{F}_n)_{n\geq 1}$ 
such that the quadratic variation $S(f)=\Big(\sum_n|df_n|^2\Big)^{1/2}$ belongs to $L_1(\Omega)$, 
and by $\h_1(\Omega)$ the space of martingales $f$ 
such that the conditioned quadratic variation $s(f)=\Big(\sum_n\mathcal{E}_{n-1}|df_n|^2\Big)^{1/2}$ belongs to $L_1(\Omega)$. 
We say that a martingale $f=(f_n)_{n\geq 1}$ is predictable in $L_1$ if there exists 
an adapted sequence $(\lambda_n)_{n\geq 0}$ of non-decreasing, non-negative functions 
such that $|f_n|\leq \lambda_{n-1}$ for all $n\geq 1$ and such that $\sup_n \lambda_n \in L_1(\Omega)$. 
We denote by $\P_1(\Omega)$ the space of all predictable martingales. 
In a disguised form in the proof of Theorem $A_\infty$ in \cite{H}, 
Herz establishes an atomic description of $\P_1(\Omega)$. 
Since $\P_1(\Omega)=\H_1(\Omega)$ for regular martingales, 
this gives an atomic decomposition of $\H_1(\Omega)$ in the regular case. 
Such a decomposition is still valid in the general case 
but for $\h_1(\Omega)$ instead of $\H_1(\Omega)$, as shown by Weisz \cite{W1}.

In this paper, we will present the noncommutative version of 
atoms and prove that atomic decomposition for the Hardy spaces of
noncommutative martingales is valid for these atoms. 
Since there are two kinds of Hardy spaces, i.e., the column and
row Hardy spaces in the noncommutative setting, we need to define
the corresponding two type atoms. This is a main difference from
the commutative case, but can be done by considering the right and
left supports of martingales as being operators on Hilbert spaces.
Roughly speaking, replacing the supports of atoms in the above (ii)
by the right (resp. left) supports we obtain the concept of
noncommutative right (resp. left) atoms, which are proved to be
suitable for the column (resp. row) Hardy spaces. On the
other hand, due to the noncommutativity some basic constructions
based on stopping times for classical martingales are not valid
in the noncommutative setting, our approach to the atomic
decomposition for the conditioned Hardy spaces of noncommutative martingales is
via the $\h_1-\bmo$ duality. 
Recall that the duality equality $(\h_1)^*=\bmo$ was established independently by \cite{JMe} and \cite{P}.
However, this method does not give an explicit atomic decomposition. 

The other main result of this paper concerns the interpolation of the conditioned Hardy spaces $\h_p$. 
Such kind of interpolation results involving Hardy spaces of noncommutative martingales 
first appear in Musat's paper \cite{M} for the spaces $\H_p$. 
We will present an extension of these results to the conditioned case. 
Note that our method is much simpler and more elementary than Musat's arguments. 
It seems that even in the commutative case, our method is simpler 
than all existing approaches to the interpolation of Hardy spaces of martingales. 
The main idea is inspired by an equivalent quasinorm for $\h_p, 0<p\leq 2$ introduced by Herz \cite{H2} in the commutative case. 
We translate this quasinorm to the noncommutative setting to obtain 
a new characterization of $\h_p$, $0<p\leq 2$, which is more convenient for interpolation. 
By this way we show that $(\bmo, \h_1)_{1/p}=\h_p$ for any $1<p<\infty$.

The study of the Hardy spaces of noncommutative martingales $\H_p$ and $\h_p$ 
in the discrete case is the starting point for the development of an $\H_p$-theory for continuous time. 
In a forthcoming paper by Marius Junge and the third named author, 
it appears that the spaces $\h_p$ are much easier to be handled than $\H_p$. 
It seems that their use is unavoidable for problems on the spaces $\H_p$ at the continuous time.

The remainder of this paper is divided into four sections. 
In Section $1$ we present some preliminaries and notation on the
noncommutative $L_p$-spaces and various Hardy spaces of
noncommutative martingales. 
The atomic decomposition of the conditioned Hardy space $\h_1(\M)$ is presented in Section $2$, 
from which we deduce the atomic decomposition of the Hardy space $\H_1(\M)$ by Davis' decomposition. 
In Section $3$ we define an equivalent quasinorm for $\h_p(\M), 0<p\leq 2$, 
and discuss the description of the dual space of $\h_p(\M), 0<p\leq 1$. 
Finally, using the results of Section $3$, the interpolation results between $\bmo$ and $\h_1$ are proved in Section $4$.

Any notation and terminology not otherwise explained, are as used in
\cite{T} for theory of von Neumann algebras, and in \cite{PX2} for
noncommutative $L_p$-spaces. Also, we refer to a recent book by Xu
\cite{X2} for an up-to-date exposition of theory of noncommutative
martingales.

\section{Preliminaries and notations}

Throughout this paper, $\mathcal{M}$ will always denote a von
Neumann algebra with a normal faithful normalized trace $\tau.$ For each
$0 < p \leq \infty,$ let $L_p (\mathcal{M}, \tau )$ or simply $L_p
(\mathcal{M})$ be the associated noncommutative $L_p$-spaces. 
We refer to \cite{PX2} for more details and historical references on these spaces. 

For $x \in L_p (\mathcal{M})$ we denote by $r(x)$ and $l(x)$ the right and left supports of $x$, respectively. 
Recall that if $x = u |x|$ is the polar decomposition of $x$, then $r(x) = u^* u$ and $l(x) = u u^*$. 
$r(x)$ (resp. $l(x)$) is also the least projection $e$ such that $x e =x$ (resp. $e x =x$). 
If $x$ is selfadjoint, $r(x) = l(x)$. 

Let us now recall the general setup for noncommutative martingales.
In the sequel, we always denote by $(\mathcal{M}_n)_{n \geq 1}$ an
increasing sequence of von Neumann subalgebras of $\mathcal{M}$
such that the union of $\mathcal{M}_n$'s is $\mathrm{w}^*$-dense in
$\mathcal{M}$ and $\mathcal{E}_n$ the conditional expectation of
$\mathcal{M}$ with respect to $\mathcal{M}_n.$

A sequence $x = (x_n)$ in $L_1(\mathcal{M})$ is called a {\it
noncommutative martingale} with respect to $(\mathcal{M}_n)_{n \geq
1}$ if $\mathcal{E}_n (x_{n+1}) = x_n$ for every $n \geq 1.$

If in addition, all $x_n$'s are in $L_p(\mathcal{M})$ for some $1
\leq p \leq \infty,$ $x$ is called an $L_p$-martingale.
In this case we set\begin{equation*}\| x \|_p = \sup_{n \geq 1} \|
x_n \|_p.\end{equation*}If $\| x \|_p < \infty,$ then $x$ is called
a bounded $L_p$-martingale.

Let $x = (x_n)$ be a noncommutative martingale with respect to
$(\mathcal{M}_n)_{n \geq 1}.$ Define $dx_n = x_n - x_{n-1}$ for $n
\geq 1$ with the usual convention that $x_0 =0.$ The sequence $dx =
(dx_n)$ is called the {\it martingale difference sequence} of $x.$
$x$ is called a {\it finite martingale} if there exists $N$ such
that $d x_n = 0$ for all $n \geq N.$ 
In the sequel, for any operator $x\in L_1(\M)$ we denote $x_n=\E_n(x)$ for $n\geq 1$.

Let us now recall the definitions of the square functions and Hardy spaces for noncommutative martingales. 
Following \cite{PX1}, we introduce the column and row versions of square functions 
relative to a (finite) martingale $x = (x_n)$:
$$S_{c,n} (x) = \Big ( \sum^n_{k = 1} |dx_k |^2 \Big )^{1/2}, \quad 
S_c (x) = \Big ( \sum^{\infty}_{k = 1} |dx_k |^2 \Big )^{1/2};$$
and
$$S_{r,n} (x) = \Big ( \sum^n_{k = 1} | dx^*_k |^2 \Big )^{1/2}, \quad
S_r (x) = \Big ( \sum^{\infty}_{k = 1} | dx^*_k |^2 \Big)^{1/2}.$$
Let $1 \leq p < \infty$. 
Define $\mathcal{H}_p^c (\mathcal{M})$
(resp. $\mathcal{H}_p^r (\mathcal{M})$) as the completion of all
finite $L_p$-martingales under the norm $\| x \|_{\mathcal{H}_p^c}=\| S_c (x) \|_p$
(resp. $\| x \|_{\mathcal{H}_p^r}=\| S_r (x) \|_p $). 
The Hardy space of noncommutative martingales is defined as
follows: if $1 \leq p < 2,$
\begin{equation*}
\mathcal{H}_p(\mathcal{M}) 
= \mathcal{H}_p^c (\mathcal{M}) + \mathcal{H}_p^r(\mathcal{M})
\end{equation*}
equipped with the norm
\begin{equation*}
\| x \|_{\mathcal{H}_p} = 
\inf \big \{ \| y\|_{\mathcal{H}_p^c} + \| z \|_{\mathcal{H}_p^r} \big\},
\end{equation*}
where the infimum is taken over all 
$y \in\mathcal{H}_p^c (\mathcal{M} )$ and $z \in \mathcal{H}_p^r(\mathcal{M} )$ 
such that $x = y + z.$ 
For $2 \leq p <\infty,$
\begin{equation*}
\mathcal{H}_p (\mathcal{M}) =
\mathcal{H}_p^c (\mathcal{M}) \cap \mathcal{H}_p^r(\mathcal{M})
\end{equation*}
equipped with the norm
\begin{equation*}
\| x \|_{\mathcal{H}_p} = 
\max \big \{ \| x\|_{\mathcal{H}_p^c} , \| x \|_{\mathcal{H}_p^r} \big\}.
\end{equation*}
The reason that $\mathcal{H}_p (\mathcal{M})$ is
defined differently according to $1 \leq p < 2$ or $2 \leq p  \leq \infty$ is presented in \cite{PX1}. 
In that paper Pisier and Xu prove the noncommutative Burkholder-Gundy inequalities 
which imply that
$\mathcal{H}_p (\mathcal{M}) = L_p (\mathcal{M})$ with equivalent
norms for $1 < p < \infty.$

We now consider the conditioned version of $\H_p$ developed in \cite{JX1}. 
Let $x = (x_n)_{n \geq 1}$ be a finite martingale in $L_2(\M)$. 
We set
$$s_{c,n} (x) = \Big ( \sum^n_{k = 1} \E_{k-1}|dx_k |^2 \Big )^{1/2}, \quad
s_c (x) = \Big ( \sum^{\infty}_{k = 1} \E_{k-1}|dx_k |^2 \Big )^{1/2};$$
and
$$s_{r,n} (x) = \Big ( \sum^n_{k = 1} \E_{k-1}| dx^*_k |^2 \Big )^{1/2}, \quad
s_r (x) = \Big ( \sum^{\infty}_{k = 1} \E_{k-1}| dx^*_k |^2 \Big)^{1/2}.$$
These will be called the column and row conditioned square functions, respectively. 
Let $0< p < \infty$. 
Define $\h_p^c (\mathcal{M})$ (resp. $\h_p^r (\mathcal{M})$) as the completion of all
finite $L_\infty$-martingales under the (quasi)norm $\| x \|_{\h_p^c}=\| s_c (x) \|_p$
(resp. $\| x \|_{\h_p^r}=\| s_r (x) \|_p $). 
For $p=\infty$, we define $\h_\infty^c(\M)$ (resp. $\h_\infty^r(\M)$) 
as the Banach space of the $L_{\infty}(\M)$-martingales $x$ such that 
$\sum_{k \geq 1} \E_{k-1}|dx_k|^2$ 
(respectively $\sum_{k \geq 1} \E_{k-1}|dx_k^*|^2$) converge for the weak operator topology. 

We also need $\ell_p(L_p(\M))$, the space of all sequences $a=(a_n)_{n\geq 1}$ in $L_p(\M)$ such that
$$\|a\|_{\ell_p(L_p(\M))}=\Big(\sum_{n\geq 1}\|a_n\|_p^p\Big)^{1/p} <\infty 
\quad \mbox{ if } 0< p <\infty,$$ 
and 
$$\|a\|_{\ell_\infty(L_\infty(\M))} = \sup_n \| a_n\|_{\infty} \quad \mbox{ if } p = \infty.$$
Let $\h_p^d(\M)$ be the subspace of $\ell_p(L_p(\M))$ consisting of all martingale difference sequences.

We define the conditioned version of martingale Hardy spaces as follows: 
If $0< p < 2,
$\begin{equation*}\h_p
(\mathcal{M}) = \h_p^d (\mathcal{M}) +  \h_p^c (\mathcal{M})
+  \h_p^r (\mathcal{M})
\end{equation*}
equipped with the (quasi)norm
\begin{equation*}
\| x \|_{\h_p} = \inf \big \{ \| w \|_{ \h_p^d} + \| y
\|_{ \h_p^c} + \| z \|_{ \h_p^r} \big \},
\end{equation*}
where the infimum is taken over all $w \in  \h_p^d
(\mathcal{M}), y \in  \h_p^c (\mathcal{M})$ and $ z \in \h_p^r
(\mathcal{M})$ such that $ x = w + y + z.$ 
For $2 \leq p <\infty,
$\begin{equation*}\h_p (\mathcal{M}) =  \h_p^d
(\mathcal{M}) \cap  \h_p^c (\mathcal{M}) \cap  \h_p^r (\mathcal{M})
\end{equation*}equipped with the norm\begin{equation*}
\| x \|_{\h_p} = \max \big \{ \| x \|_{ \h_p^d}, \|x
\|_{ \h_p^c}, \| x \|_{ \h_p^r} \big \}.
\end{equation*}
The noncommutative Burkholder inequalities proved in \cite{JX1} state that 
\begin{equation}\label{eq:burkh}
\h_p (\mathcal{M}) = L_p(\mathcal{M})
\end{equation}
with equivalent norms for all $1 < p < \infty.$

In the sequel, $(\M_n)_{n\geq 1}$ will be a filtration of von Neumann subalgebras of $\M$. 
All martingales will be with respect to this filtration.

\section{Atomic decompositions}

Let us now introduce the concept of noncommutative atoms.

\begin{definition}\label{df:atom} 
$a \in L_2 (\mathcal{M})$ is said to be a $(1,2)_c$-atom
with respect to $(\mathcal{M}_n)_{n \geq 1},$ if there exist $n \geq
1$ and a projection $e \in \mathcal{M}_n$ such that
\begin{enumerate}[{\rm (i)}]

\item  $\mathcal{E}_n (a) =0;$

\item  $r (a) \leq e;$

\item  $ \| a \|_2 \leq \tau (e)^{- 1/2}.$
\end{enumerate}
Replacing $\mathrm{(ii)}$ by $ \mathrm{(ii)'}~~l (a)
\leq e,$ we get the notion of a $(1,2)_r$-atom.
\end{definition}

Here, $(1,2)_c$-atoms and $(1,2)_r$-atoms are
noncommutative analogues of $(1,2)$-atoms for classical martingales. 
In a later remark we will discuss the noncommutative analogue of $(p,2)$-atoms.
These atoms satisfy the following useful estimates.

\begin{proposition}\label{pr:atom_properties} 
If $a $ is a $(1,2)_c$-atom then 
$$ \| a \|_{\mathcal{H}_1^c} \leq 1 
\quad \mbox{ and } \quad
\| a \|_{\h_1^c} \leq 1.$$
The similar inequalities hold for $(1,2)_r$-atoms.
\end{proposition}

\begin{proof} Let $e$ be a projection associated with $a$
satisfying $\mathrm{(i)-(iii)}$ of Definition \ref{df:atom}. 
Let $a_k=\E_k(a)$. 
Observe that $a_k = 0$ for $k \leq n$, so $da_k=0$ for $k \leq n$. 
For $k \geq n + 1$ we have 
$$\begin{array}{lclcl}
e |d a_k|^2 
&= &[\mathcal{E}_{k}(ea^*)-\mathcal{E}_{k-1}(e a^*) ] da_k &=& | da_k |^2 \\
&= &da^*_k [ \mathcal{E}_{k}(a e)-\mathcal{E}_{k-1}( a e) ]&=& |d a_k |^2e.
\end{array}$$
This gives 
\begin{equation*}
e |d a_k |^2 = | d a_k |^2 = | d a_k |^2 e 
\end{equation*}
for any $k\geq 1.$ 
Hence, we obtain
\begin{equation*}
e S_c ( a ) = S_c( a ) = S_c ( a ) e.
\end{equation*}
Consequently, the noncommutative H\"{o}lder inequality implies
\begin{equation*} \|a \|_{\mathcal{H}_1^c} =  \tau [ e S_c ( a) ]
\leq \| S_c ( a) \|_2 \| e \|_2 
=\|  a \|_2 \| e \|_2 \leq 1.\end{equation*} 
Since $e\in \M_n$, for $k \geq n+1$ we have 
$$\begin{array}{lclcl}
e \mathcal{E}_{k-1} (| d a_k |^2 )
&=&\mathcal{E}_{k-1} (e | d a_k |^2 ) 
&=& \mathcal{E}_{k-1} (| d a_k |^2)\\
&=& \mathcal{E}_{k-1} ( | d a_k |^2 e ) 
&=& \mathcal{E}_{k-1} ( | d a_k|^2 ) e.
\end{array}$$
Thus, we deduce
$$\| a\|_{\h_1^c} \leq 1.$$
\end{proof}

Now, atomic Hardy spaces are defined as follows.

\begin{definition}\label{df:h1at} 
We define $\h_1^{c, \mathrm{at}} (\mathcal{M} )$ as 
the Banach space of all $x \in L_1 (\mathcal{M} )$ which admit a decomposition  
$$x = \sum_k \lambda_k a_k$$
with for each $k$, $a_k$ a $(1,2)_c$-atom or an element in $L_1(\M_1)$ of norm $\leq 1$, 
and $\lambda_k \in \mathbb{C}$ satisfying $\sum_k |\lambda_k| < \infty$. 
We equip this space with the norm
$$\| x \|_{\h_1^{c, \mathrm{at}}} = \inf \sum_k | \lambda_k |,$$
where the infimum is taken over all decompositions of $x$ described above.

Similarly, we define $\h_1^{r,\mathrm{at}}(\mathcal{M} )$ and $\|\cdot\|_{\h_1^{r, \mathrm{at}}}$.
\end{definition}

It is easy to see that  $\h_1^{c,\mathrm{at}}(\M)$ is a Banach space. 
By Proposition \ref{pr:atom_properties} we have the contractive inclusion $\h_1^{c,\at}(\M)\subset \h_1^c(\M)$. 
The following theorem shows that these two spaces coincide. 
That establishes the atomic decomposition of the conditioned Hardy space $\h_1^c(\M)$. 
This is the main result of this section.

\begin{theorem}\label{th:c_atomic_decomp}
We have 
$$\h_1^c ( \mathcal{M} ) = \h_1^{c, \mathrm{at}} ( \mathcal{M} ) \quad \mbox{ with equivalent norms. }$$
More precisely, if $x \in \h_1^c (\mathcal{M})$
\begin{equation*}
\frac{1}{\sqrt{2}} \| x \|_{\h_1^{c, \mathrm{at}}} \leq \| x
\|_{\h_1^c} \leq  \| x \|_{\h_1^{c, \mathrm{at}}}.
\end{equation*}
Similarly, $\h_1^r ( \mathcal{M} ) = \h_1^{r, \mathrm{at}} (
\mathcal{M} )$ with the same equivalence constants.
\end{theorem}

We will show the remaining inclusion $\h_1^c(\M)\subset\h_1^{c,\at} (\M)$ by duality. 
Recall that the dual space of $\h_1^c(\M)$ is the space $\bmo^c(\M)$ defined as follows 
(we refer to \cite {JMe} and \cite{P} for details). 
Let
$$\bmo^c(\mathcal{M})
= \big \{ x \in L_2 (\mathcal{M}): 
\sup_{n \geq 1} \|\mathcal{E}_n | x - x_n|^2 \|_{\infty} < \infty \big\}$$
and equip $\bmo^c(\M)$ with the norm 
$$\| x \|_{\bmo^c } 
= \max \Big( \|\mathcal{E}_1 (x) \|_{\infty}\; , \;\sup_{n \geq 1} \|\mathcal{E}_n | x - x_n |^2 \|^{1/2}_{\infty}\Big).$$
This is a Banach space. 
Similarly, we define the row version $\bmo^r(\M)$. 
Since $x_n=\E_n(x)$, we have 
$$\E_n|x-x_n|^2=\E_n|x|^2-|x_n|^2\leq \E_n|x|^2.$$
Thus the contractivity of the conditional expectation yields
\begin{equation}\label{eq:bmo_Linfty}
 \| x \|_{\bmo^c }\leq \|x\|_\infty.
\end{equation}

We will describe the dual space of $\h^{c,\at}_1(\M)$ as a noncommutative Lipschitz space defined as follows. 
We set 
$$\Lambda^c(\M)=\big\{x\in L_2(\M):\|x\|_{\Lambda^c}<\infty\big\}$$
with 
$$\|x\|_{\Lambda^c}
=\max \Big(\|\E_1(x)\|_\infty \;, \; 
\sup_{n \geq 1} \sup_{e \in \mathcal{P}_n} \tau (e )^{-1/2}\tau \big ( e |x - x_n |^2 \big)^{1/2}\Big),$$
where $\mathcal{P}_n$ denotes the lattice of projections of $\mathcal{M}_n.$ 
Similarly, we define 
$$\Lambda^r(\M)=\big\{x\in L^2(\M):x^*\in \Lambda^c(\M)\big\}$$
equipped with the norm 
$$\|x\|_{\Lambda^r}=\|x^*\|_{\Lambda^c}.$$
The relation between Lipschitz space and $\bmo$ space can be stated as follows. 

\begin{proposition}\label{pr:bmo=lambda}
We have $\bmo^c(\M)=\Lambda^c(\M)$ and $\bmo^r(\M)=\Lambda^r(\M)$ isometrically. 
\end{proposition}

\begin{proof}
Let $x\in \bmo^c(\M)$. 
It is obvious that by the noncommutative H\"{o}lder inequality we have, for all $n\geq1$,
$$\sup_{e \in \mathcal{P}_n} \tau (e )^{-1/2}\tau \big ( e |x - x_n |^2 \big)^{1/2}
\leq \|\E_n|x-x_n|^2\|_\infty^{1/2}.$$
To prove the reverse inclusion, by duality we can write 
$$\begin{array}{ccl}
\| \mathcal{E}_n | x - x_{n}|^2 \|_{\infty} & = & 
\displaystyle\sup_{\| y \|_1 \leq 1, \; y \in L_1^+(\mathcal{M}_n)} \big | \tau ( y |x - x_{n} |^2) \big | \\
& = & \displaystyle\sup_{e \in \mathcal{P}_n} \tau (e )^{-1} \tau (e |x -x_{n}|^2 ),
\end{array}$$
where the last equality comes from 
the density of linear combinations of mutually disjoint projections in $L_1(\mathcal{M}_n)$. 
Thus $\|x\|_{\Lambda^c}=\|x\|_{\bmo^c}$, and the same holds for the row spaces.
\end{proof}

We now turn to the duality between the conditioned atomic space $\h_1^{c,\at}(\M)$ 
and the Lipschitz space $\Lambda^c(\M)$.

\begin{theorem}\label{th:duality_h1at}
We have 
$h_1^{c, \mathrm{at}}( \mathcal{M} )^* = \Lambda^c ( \mathcal{M} )$ isometrically. 
More precisely,
\begin{enumerate}[{\rm (i)}]
\item  Every $x \in \Lambda^c ( \mathcal{M} )$ defines a
continuous linear functional on $\h_1^{c, \mathrm{at}} ( \mathcal{M})$ by
\begin{equation}\label{eq:duality_bracket_at} 
\varphi_x ( y) = \tau (x^*y),\quad \forall y \in L_2 (\mathcal{M}).
\end{equation}
\item  Conversely, each $\varphi \in \h_1^{c, \mathrm{at}} (\mathcal{M} )^*$ is given as \eqref{eq:duality_bracket_at} 
by some $x \in \Lambda^c (\mathcal{M} ).$
\end{enumerate}
Similarly, $\h_1^{r, \mathrm{at}} ( \mathcal{M} )^* = \Lambda^r (\mathcal{M} )$ isometrically.
\end{theorem}

\begin{rk}
Remark that we have defined the duality bracket \eqref{eq:duality_bracket_at} for operators in $L_2(\M)$. 
This is sufficient for $L_2(\M)$ is dense in $\h_1^{c, \at}(\M)$. 
The latter density easily follows from the decomposition $L_2(\M)=L^0_2(\M)\oplus L_2(\M_1)$, 
where $L^0_2(\M)=\{x\in L_2(\M):\E_1(x)=0\}$.
\end{rk}

\hspace{-0.4cm}{\it Proof of Theorem \ref{th:duality_h1at}.}
We first show $\Lambda^c (\M)\subset  h_1^{c,\at}(\M)^* $. 
In fact we will not need this inclusion for the proof of Theorem \ref{th:c_atomic_decomp}, 
however we include the proof for the sake of completeness. 
Let $x \in \Lambda^c(\M)$.  
For any $(1,2)_c$-atom $a$ 
associated with a projection $e$ satisfying $\mathrm{(i)-(iii)}$ of Definition \ref{df:atom}, by the
noncommutative H\"{o}lder inequality we have
$$\begin{array}{ccl}
\big | \tau ( x^* a )\big | 
& = &\big | \tau ( (x - x_{n})^*ae) \big |\\
& \leq&  \|e(x -x_n)^*\|_2\| a \|_2\\
& \leq &\tau (e)^{- 1/2} \big [ \tau ( e | x -x_n|^2 ) \big ]^{1/2}\\
& \leq &\| x \|_{\Lambda^c}.
\end{array}$$
On the other hand, for any $a\in L_1(\M_1)$ with $\|a\|_1\leq 1$ we have
$$|\tau(x^*a)|=|\tau(\E_1(x)^*a)|
\leq \|\E_1(x)\|_\infty\|a\|_1
\leq \|x\|_{\Lambda^c}.$$
Thus, we deduce that 
\begin{equation*} 
\big | \tau (x^* y ) \big | \leq \| x \|_{\Lambda^c} \| y\|_{\h_1^{c, \mathrm{at}}}
\end{equation*}
for all $y \in L_2(\mathcal{M}).$ 
Hence, $\varphi_x $ extends to a continuous functional on $ \h_1^{c,\at}(\M)$ of norm less than or equal to $\|x\|_{\Lambda^c}.$ 

Conversely, let $\varphi \in \h_1^{c,\at} (\M)^*.$ 
As explained in the previous remark, $L_2(\M)\subset \h_1^{c,\at}(\M)$ 
so by the Riesz representation theorem there exists $x \in L_2(\M)$ such that
$$\varphi(y) = \tau (x^* y), \quad \forall y \in L_2 (\mathcal{M}).$$
Fix $n\geq 1$ and let $e \in \mathcal{P}_n$. 
We set
$$y_e = \frac{(x- x_n )e}{ \| (x -x_n ) e \|_2 \tau (e)^{1/2}}.$$
It is clear that $y_e$ is a $(1,2)_c$-atom with the associated projection $e$. 
Then
$$\| \varphi \| 
\geq | \varphi(y_e) | 
= | \tau ( (x-x_n)^* y_e)| 
= \frac{1}{\tau (e)^{1/2}} \big [
\tau(e |x -x_n|^2 ) \big ]^{1/2}.$$
On the other hand, let $y\in L_1(\M_1), \|y\|_1\leq 1$ 
be such that $\|\E_1(x)\|_\infty=|\tau(x^*y)|$. 
Then $\|\E_1(x)\|_\infty\leq \| \varphi \|$. 
Combining these estimates we obtain $\|x\|_{\Lambda^c}\leq \|\varphi\|$. 
This ends the proof of the duality $(\h_1^{c,\at}(\M))^*=\Lambda^c(\M)$. 
Passing to adjoints yields the duality $(\h_1^{r,\at}(\M))^*=\Lambda^r(\M)$.
\cqd

\vspace{0.3cm}

We can now prove the reverse inclusion of Theorem \ref{th:c_atomic_decomp}.

\vspace{0.3cm}

\hspace{-0.4cm}\textit{Proof of Theorem \ref{th:c_atomic_decomp}.} 
By Proposition \ref{pr:atom_properties} we already know that \\
$\h_1^{c,\at}(\M)\subset \h_1^c(\M)$. 
Combining Proposition \ref{pr:bmo=lambda} and Theorem \ref{th:duality_h1at} 
we obtain that $(\h_1^{c,\at}(\M))^*=\bmo^c(\M)$ with equal norms. 
The duality between $\h_1^c(\M)$ and $\bmo^c(\M)$ proved in \cite{JMe} and \cite{P} then yields that
$(\h_1^{c,\at}(\M))^*=( \h_1^c(\M))^*$ with the following equivalence constants
$$\frac{1}{\sqrt{2}}\| \varphi_x \|_{(\h_1^c)^*}
\leq  \|x\|_{\bmo^c} 
=\| \varphi_x \|_{(\h_1^{c, \mathrm{at}})^*} 
\leq \| \varphi_x \|_{(\h_1^c)^*}.$$
This ends the proof of Theorem \ref{th:c_atomic_decomp}.
\cqd

\vspace{0.3cm}

We can generalize this decomposition to the whole space $\h_1(\M)$. 
To this end we need the following definition.

\begin{definition} 
We set 
$$\h_1^{\at}(\M)
= \h_1^d (\M) + \h_1^{c,\at}(\M) + \h_1^{r,\at}(\M),$$
equipped with the sum norm
$$\| x\|_{\h_1^{\mathrm{at}}} = 
\inf \big \{ \|w\|_{\h_1^d} + \|y \|_{\h_1^{c,\at}} + \| z \|_{\h_1^{r,\at}} \big\},$$
where the infimum is taken over all 
$w \in\h_1^d (\mathcal{M}), y \in \h_1^{c, \mathrm{at}} (\mathcal{M} ),$ and 
$z \in \h_1^{r, \mathrm{at}} ( \mathcal{M} )$
such that $x =w + y + z.$
\end{definition}

Thus Theorem \ref{th:c_atomic_decomp} clearly implies the following.

\begin{theorem}\label{th:atomic_decomp}
We have 
$$\h_1( \mathcal{M} ) = \h_1^{\mathrm{at}} ( \mathcal{M} ) \quad \mbox{ with equivalent norms. }$$
More precisely, if $x \in \h_1(\mathcal{M})$
$$\frac{1}{ \sqrt{2}} \| x \|_{\h_1^{\mathrm{at}}} \leq \| x
\|_{\h_1} \leq  \| x \|_{\h_1^{\mathrm{at}}}.$$
\end{theorem}

The noncommutative Davis' decomposition presented in \cite{P} states that $\H_1(\M)=\h_1(\M)$. 
Thus Theorem \ref{th:atomic_decomp} yields that $\H_1(\M)=\h_1^{\at}(\M)$, 
which means that we can decompose any martingale in $\H_1(\M)$ 
in an atomic part and a diagonal part. 
This is the atomic decomposition for the Hardy space of noncommutative martingales.

\section{An equivalent quasinorm for $\h_p, 0<p\leq 2$}

In the commutative case Herz described in \cite{H2} an equivalent quasinorm for $\h_p, 0<p\leq 2$. 
This section is devoted to determining a noncommutative analogue of this. 
This characterization of $\h_p$ will be useful in the sequel. 
Indeed, this will imply an interpolation result in the next section. 
To define equivalent quasinorms of $\|\cdot\|_{\h_p^c}$ and $\|\cdot\|_{\h_p^r}$ for $0<p\leq 2$ 
we introduce the index class $W$ which consists of sequences $\{w_n\}_{n\in \N}$ 
such that $\{w_n^{2/p-1}\}_{n\in \N}$ is nondecreasing 
with each $w_n \in L_1^+(\M_n)$ invertible with bounded inverse and $\|w_n\|_1\leq 1$. 

For an $L_2$-martingale $x$ we set
$$N^c_p(x)=\inf_W \Big[\tau\Big(\sum_{n\geq0}w_n^{1-2/p}|dx_{n+1}|^2\Big)\Big]^{1/2}$$
and
$$N^r_p(x)=\inf_W \Big[\tau\Big(\sum_{n\geq0}w_n^{1-2/p}|dx_{n+1}^*|^2\Big)\Big]^{1/2}.$$

We need the following well-known lemma, and include a proof for the convenience of the reader 
(see Lemma $1$ of \cite{Ta} for the case $f(t)=t^p$).

\begin{lemma}\label{le:preliminary}
Let $f$ be a function in $C^1(\R^+)$ and $x,y \in \M^+$. Then
$$\tau(f(x+y)-f(x))=\tau\Big(\int_0^1f'(x+ty)ydt\Big).$$
\end{lemma}

\begin{proof}
Note that considering $f-f(0)$, we may assume that $f(0)=0$. 
We set $\varphi_f(t)=\tau(f(x+ty))$, for $t\in [0,1]$. Then 
\begin{equation}\label{eq:lemma}
 \varphi_f'(t)=\tau(f'(x+ty)y), \quad \forall t\in [0,1].
\end{equation}
Indeed, the tracial property of $\tau$ implies this equality for $t=0$ and $f(t)=t^n, n\in \N$, 
and we can extend this result for all $f$ polynomials by linearity. 
A translation argument gives \eqref{eq:lemma} for all $f$ polynomials. 
Finally, we generalize for all $f$ by approximation. 
Indeed, we can approximate $f'$ by a sequence $(p_n)_{n\geq 1}$ of polynomials, uniformly on the compact set $K=[0,\|x\|_\infty+\|y\|_\infty]$. 
Then the sequence of polynomials $(q_n)$ defined by $q_n(s)=\int_0^sp_n(t)dt$ for each $n\geq 1$ converges uniformly to $f$ on $K$. 
Since $(\varphi_{q_n}')$ converges to $\varphi_f'$ uniformly on $[0,1]$ (by the derivation theorem), 
we get \eqref{eq:lemma} by the finiteness of the trace. \\
Now writing $\varphi_f(1)-\varphi_f(0)=\int_0^1\varphi_f'(t)dt$ we obtain the desired result.
\end{proof}

\begin{proposition}\label{prop:eq_norm}
For $0<p\leq 2$ and $x\in L_2(\M)$ we have
\begin{equation}\label{eq:eq_norm}
\Big(\frac{p}{2}\Big)^{1/2}N^c_p(x)\leq \|x\|_{\h^c_p}\leq N^c_p(x).
\end{equation}
A similar statement holds for $\h^r_p(\M)$ and $N^r_p$.
\end{proposition}

\begin{proof}
Note that
$$\begin{array}{ccl}
N^c_p(x)&=&\displaystyle\inf_W \Big[\tau\Big(\displaystyle\sum_{n\geq0}w_n^{1-2/p}\E_n|dx_{n+1}|^2\Big)\Big]^{1/2}\\
&=&\displaystyle\inf_W \Big[\tau\Big(\displaystyle\sum_{n\geq0}w_n^{1-2/p}(s_{c,n+1}(x)^2-s_{c,n}(x)^2)\Big)\Big]^{1/2}.
\end{array}$$
Let $x\in L_2(\M)$ with $\|x\|_{\h^c_p}<1$. 
By approximation we can assume that $x\in L_\infty(\M)$ 
and $s_{c,n}(x)$ is invertible with bounded inverse for every $n\geq 1$. 
Then $\{s_{c,n+1}(x)^p\}\in W$; so
$$N^c_p(x)\leq \Big[\tau\Big(\sum_{n\geq0}s_{c,n+1}(x)^{p-2}(s_{c,n+1}(x)^2-s_{c,n}(x)^2)\Big)\Big]^{1/2}.$$
Applying Lemma \ref{le:preliminary} with $f(t)=t^{p/2}, x+y=s_{c,n+1}(x)^2$ and $x=s_{c,n}(x)^2$ we obtain
$$\begin{array}{l}
\tau(s_{c,n+1}(x)^p-s_{c,n}(x)^p)=\\
\tau\Big(\displaystyle\int_0^1\displaystyle\frac{p}{2}\big[s_{c,n}(x)^2+t(s_{c,n+1}(x)^2-s_{c,n}(x)^2)\big]^{\frac{p}{2}-1}
\big[s_{c,n+1}(x)^2-s_{c,n}(x)^2\big]dt\Big)\\
\geq
\displaystyle\frac{p}{2}\tau (s_{c,n+1}(x)^{p-2}(s_{c,n+1}(x)^2-s_{c,n}(x)^2)),
\end{array}$$
where we have used the fact that the operator function $a\mapsto a^{\frac{p}{2}-1}$ is nonincreasing for $-1<\frac{p}{2}-1\leq 0$. 
Taking the sum over $n$ leads to
$$N^c_p(x)^2\leq \frac{2}{p} \tau(s_c(x)^p)=\frac{2}{p}.$$

We turn to the other estimate. 
Given $\{w_n\}\in W$ put 
$$w^{2/p-1}=\displaystyle\lim_{n\rightarrow +\infty} w_n^{2/p-1}=\displaystyle\sup_n w_n^{2/p-1}.$$
It follows that $\{w_n^{1-2/p}\}$ decreases to $w^{1-2/p}$ and 
$$\begin{array}{ccl}
\tau\Big(\displaystyle\sum_{n\geq0}w_n^{1-2/p}|dx_{n+1}|^2\Big)
&\geq& \tau\Big(w^{1-2/p}\displaystyle\sum_{n\geq0}\E_n|dx_{n+1}|^2\Big)\\
&=&\tau\Big(w^{1-2/p}s_c(x)^2\Big).
\end{array}$$
Since $\frac{1}{p}=\frac{1}{2}+\frac{2-p}{2p}$ 
the H\"{o}lder inequality gives
$$\begin{array}{ccl}
\|s_c(x)\|_p&=&\|w^{1/p-1/2}w^{1/2-1/p}s_c(x)\|_p\\
&\leq&\|w^{1/p-1/2}\|_{2p/(2-p)}\|w^{1/2-1/p}s_c(x)\|_2\\
&=&\tau(w)^{1/p-1/2}\tau(w^{1-2/p}s_c(x)^2)^{1/2}.
\end{array}$$
Now $\tau(w)\leq 1$; so we have
$$\|s_c(x)\|_p\leq \Big[\tau\Big(\sum_{n\geq 0}w_n^{1-2/p}|dx_{n+1}|^2\Big)\Big]^{1/2}$$
for all $\{w_n\} \in W$.
\end{proof}

Thus the quasinorm $N^c_p$ is equivalent to $\|\cdot\|_{\h_p^c}$ on $L_2(\M)$. 
So $\h_p^c(\M)$ can also be defined as the completion of all finite $L_2$-martingales with respect to $N_p^c$ for $0<p\leq 2$. 
This new characterization of $\h_p^c(\M)$ yields the following description of its dual space.

\begin{theorem}\label{th:duality_hp}
Let $0<p\leq 2$ and $q$ be determined by $\frac{1}{q}=1-\frac{1}{p}$. 
Then the dual space of $\h_p^c(\M)$ coincide with the $L_2$-martingales $x$ for which 
$M^c_q(x)=\displaystyle\sup_W \Big[\tau\Big(\displaystyle\sum_{n\geq 0}w_n^{1-2/q}|dx_{n+1}|^2\Big)\Big]^{1/2}<\infty$. 
More precisely,
\begin{enumerate}[\rm(i)]
\item Every $L_2$-martingale $x$ such that $M^c_q(x)<\infty$ defines a continuous linear functional on $\h_p^c(\M)$ by
\begin{equation*}
\phi_x(y)=\tau(yx^*) \mbox{ for } y\in L_2(\M).
\end{equation*}
\item Conversely, any continuous linear functional $\phi$ on $\h_p^c(\M)$ is given as above 
by some $x$ such that $M^c_q(x)<\infty$.
\end{enumerate}
Similarly, the dual space of $\h_p^r(\M)$ coincide with the $L_2$-martingales $x$ for which 
$M^r_q(x)=M^c_q(x^*)<\infty$. 
\end{theorem}

\begin{proof}
Let $x$ be such that $M^c_q(x)<\infty$. 
Then $x$ defines a continuous linear functional on $\h_p^c(\M)$ 
by $\phi_x(y)=\tau(yx^*)$ for $y\in L_2(\M)$. 
To see this fix $\{w_n\}\in W$. The Cauchy-Schwarz inequality gives
$$\begin{array}{ccl}
\tau(yx^*)
&=&\displaystyle\sum_{n\geq 0} \tau\Big((dy_{n+1}w_n^{1/2-1/p})(dx_{n+1}w_n^{1/2-1/q})^*\Big)\\
&\leq&\Big(\displaystyle\sum_{n\geq 0}\tau(w_n^{1-2/p}|dy_{n+1}|^2)\Big)^{1/2}
\Big(\displaystyle\sum_{n\geq 0}\tau(w_n^{1-2/q}|dx_{n+1}|^2)\Big)^{1/2}\\
&\leq&\Big(\displaystyle\sum_{n\geq 0}\tau(w_n^{1-2/p}|dy_{n+1}|^2)\Big)^{1/2} M^c_q(x). 
\end{array}$$
Taking the infimum over $W$ we obtain $\tau(yx^*)\leq N^c_p(y) M^c_q(x)$.

Conversely, let $\phi$ be a continuous linear functional on $\h_p^c(\M)$ of norm $\leq 1$. 
As $L_2(\M)\subset\h_p^c(\M)$, 
$\phi$ induces a continuous linear functional on $L_2(\M)$. 
Thus there exists $x\in L_2(\M)$ such that $\phi(y)=\tau(yx^*)$ for $y\in L_2(\M)$. 
By the density of $L_2(\M)$ in $\h_p^c(\M)$ we have 
$$\|\phi\|_{(\h_p^c)^*}=\sup_{y\in L_2(\M), \|y\|_{\h^c_p}\leq 1}|\tau(yx^*)| \leq 1.$$
Thus by Proposition \ref{prop:eq_norm} we obtain
\begin{equation}\label{eq:norm_phi}
\sup_{y\in L_2(\M), N_p^c(y)\leq 1}|\tau(yx^*)| \leq 1.
\end{equation}
We want to show that $M^c_q(x)<\infty$. 
Fix $\{w_n\}\in W$. 
Let $y$ be the martingale defined by $dy_{n+1}=dx_{n+1}w_n^{1-2/q}, \forall n\in \N$. 
By $\eqref{eq:norm_phi}$ we have
$$\begin{array}{ccl}
\tau(yx^*)&=&\tau\Big(\displaystyle\sum_{n\geq 0}w_n^{1-2/q}|dx_{n+1}|^2\Big)\leq N^c_p(y)\\
&\leq& \tau\Big(\displaystyle\sum_{n\geq 0}w_n^{1-2/q}|dx_{n+1}|^2\Big)^{1/2}.
\end{array}$$
Thus
$$\tau\Big(\displaystyle\sum_{n\geq 0}w_n^{1-2/q}|dx_{n+1}|^2\Big)\leq 1, \quad \forall \{w_n\}\in W.$$
Taking the supremum over $W$ we obtain $M^c_q(x)\leq 1$. 

Passing to adjoints yields the description of the continuous linear functionals on $\h_p^r(\M)$.
\end{proof}

Remark that for $-\infty<1/q\leq1/2$, $M^c_q$ and $M^r_q$ define two norms. 
Let $X^c_q$ (resp. $X^r_q$) be the Banach space consisting of 
the $L_2$-martingales $x$ for which $M^c_q(x)$ (resp. $M^r_q(x)$) is finite. 
Theorem \ref{th:duality_hp} shows that $(\h^c_p(\M))^*=X^c_q$ and $(\h^r_p(\M))^*=X^r_q$ 
for $0<p\leq 2$, $\frac{1}{q}=1-\frac{1}{p}$. 

For $-\infty<1/q\leq1/2$, note that $M^c_q(x)$ can be rewritten in the following form. 
Given $\{w_n\}_{n\geq 0}\in W$ we put 
$$g_n=(w_n^{2/s}-w_{n-1}^{2/s})^{1/2}, \quad \forall n \geq 1$$
where $\frac{1}{s}=\frac{1}{2}-\frac{1}{q}$. 
It is clear that
$$\{g_n\}_{n\geq 1}\in G=\Big\{\{h_n\}_{n\geq 1}; h_n \in L_s(\M_n), \tau\Big(\Big(\sum_{n\geq 1}|h_n|^2\Big)^{s/2}\Big) \leq 1\Big\}.$$
Then 
$$M^c_q(x)=\sup_G \Big[ \tau\Big(\sum_{n\geq 1}|g_n|^2\E_n|x-x_n|^2\Big)\Big]^{1/2}.$$
It is now easy to see that the dual form of Junge's noncommutative Doob maximal inequality (\cite{J}) implies that 
for $q\geq 2, X^c_q=L^c_q\mo(\M)$ with equivalent norms, 
where $L^c_q\mo(\M)$ is defined in \cite{P}.

Similarly, we have $X^r_q=L^r_q\mo(\M)$ with equivalent norms.

Thus for $1\leq p \leq 2$, Theorem \ref{th:duality_hp} gives another proof of the duality  obtained in \cite{P} 
between $\h_p(\M)$ and $L_q\mo(\M)$ for $\frac{1}{p}+\frac{1}{q}=1$. 
Note that this new proof is much simpler and yields a better constant for the upper estimate, 
that is $\sqrt{p/2}$ instead of $\sqrt{2}$.

For $0<p<1$, Theorem \ref{th:duality_hp} leads to a first description 
of the dual space of $\h_p(\M)$. 
However, this description is not satisfactory. 
Following the classical case, we would like to describe this dual space as the Lipschitz space 
$\Lambda^c_\alpha(\M)$ defined in the previous section as the dual space of $\h_p^{c,\at}(\M)$. 
Thus the description of the dual space of $\h_p(\M)$ for $0<p<1$ is closely related to 
the atomic decomposition of $\h_p(\M)$.

\section{Interpolation of $\h_p$ spaces}

It is a rather easy matter to identify interpolation spaces between commutative or noncommutative 
$L_p$-spaces by real or complex method.
However, we need more efforts to establish interpolation results between Hardy spaces of martingales 
(see \cite{JJ}, and also \cite{X3}). 
Musat (\cite{M}) extended Janson and Jones' interpolation theorem for Hardy spaces of martingales 
to the noncommutative setting. 
She proved in particular that for $1\leq q<q_\theta<\infty$
\begin{equation}\label{eq:interpolation_musat}
(\BMO^c(\M),\H^c_q(\M))_{\frac{q}{q_\theta}}=\H^c_{q_\theta}(\M).
\end{equation}
See also \cite{JM} for a different proof with better constants.
This section is devoted to showing the analogue of \eqref{eq:interpolation_musat} in the conditioned case. 
Our approach is simpler and more elementary than Musat's and also valid for her situation.

We refer to \cite{BL} for details on interpolation. 
Recall that the noncommutative $L_p$-spaces associated with a semifinite von Neumann algebra 
form interpolation scales with respect to the complex method and the real method.
More precisely, for $0<\theta<1, 1\leq p_0< p_1\leq\infty$ and $1\leq q_0,q_1,q\leq \infty$ we have
\begin{equation}\label{eq:cinterLp}
L_p(\M)=(L_{p_0}(\M),L_{p_1}(\M))_\theta \quad \mbox{ (with equal norms)}
\end{equation}
and
\begin{equation}\label{eq:rinterLpq}
L_{p,q}(\M)=(L_{p_0,q_0}(\M),L_{p_1,q_1}(\M))_{\theta,q} \quad \mbox{ (with equivalent norms)}
\end{equation}
where $\frac{1}{p}=\frac{1-\theta}{p_0}+\frac{\theta}{p_1}$, 
and where $L_{p,q}(\M)$ denotes the noncommutative Lorentz space on $(\M,\tau)$.

We can now state the main result of this section which deals with complex interpolation between 
the column spaces $\bmo^c(\M)$ and $\h^c_1(\M)$.

\begin{theorem}\label{th:interp_bmoc_h1c}
Let $1< p<\infty$.
Then, the following holds with equivalent norms
\begin{equation}\label{eq:interpolation}
(\bmo^c(\M), \h^c_1(\M))_{\frac{1}{p}}=\h^c_{p}(\M).
\end{equation}
\end{theorem}

\begin{rk}\label{rk:density}
All spaces considered here are compatible in the sense that 
they can be embedded in the $*$-algebra of measurable operators with respect to 
$(\M\overline{\otimes} \mathrm{B}(\ell_2(\N^2)),\tau\otimes \mathrm{Tr})$. 
Indeed, for each $1\leq p<\infty$, $\h^c_p(\M)$ can be identified with a subspace of $L_p(\M\overline{\otimes}\mathrm{B}(\ell_2(\N^2)))$. 
Recall that $\h^c_p(\M)$ is also defined as the closure in $L_p^{\mathrm{cond}}(\M;\ell^c_2)$ 
of all finite martingale differences in $\M$. 
Here $L_p^{\mathrm{cond}}(\M;\ell^c_2)$ is the subspace of $L_p(\M,\ell^c_2(\N^2))$ introduced by Junge \cite{J} 
consisting of all double indexed sequences $(x_{nk})$ such that $x_{nk}\in L_p(\M_n)$ for all $k\in \N$. 
We refer to \cite{PX1} for details on the column and row spaces $L_p(\M,\ell^c_2)$ and $L_p(\M,\ell^r_2)$. 
Furthermore, by the H\"{o}lder inequality and duality, recalling that the trace is finite, 
we have, for $1\leq p<q<\infty$, the continuous inclusions
\begin{equation*}
 L_\infty(\M)\subset \bmo^c(\M) \subset \h_q^c(\M) \subset \h_p^c(\M).
\end{equation*}
The first inclusion is proved by \eqref{eq:bmo_Linfty}. 
The second one comes from the third one by duality. 
Indeed, it is  proved in \cite{JX1} that for $1<p<\infty$ and $\frac{1}{p}+\frac{1}{p'}=1$, 
we have $(\h^c_{p}(\M))^*=\h^c_{p'}(\M)$, 
and, as already mentioned above, we have $(\h^c_{1}(\M))^*=\bmo^c(\M)$ (see \cite{P}). 
Note that $ L_\infty(\M)$ is dense in all spaces above, except $ \bmo^c(\M)$. 
This implies that $\bmo^c(\M)$ and $\h_q^c(\M)$ are dense in $\h_p^c(\M)$ for $1\leq p<q<\infty$.
\end{rk}

We will need Wolff's interpolation theorem (see \cite{Wo}). 
This result states that given Banach spaces $E_i$ $(i=1,2,3,4)$ 
such that $E_1\cap E_4$ is dense in both $E_2$ and $E_3$, and
$$E_2=(E_1,E_3)_\theta \quad \mbox{ and } \quad E_3=(E_2,E_4)_\phi$$
for some $0<\theta,\phi<1$, then
\begin{equation}\label{eq:wolff}
E_2=(E_1,E_4)_\varsigma \quad \mbox{ and } \quad E_3=(E_1,E_4)_\xi,
\end{equation}
where $\varsigma=\frac{\theta\phi}{1-\theta+\theta\phi}$ 
and $\xi=\frac{\phi}{1-\theta+\theta\phi}$. 
The main step of the proof of Theorem \ref{th:interp_bmoc_h1c} is the following lemma 
which is based on the equivalent quasinorm $N^c_p$ of $\|\cdot\|_{\h^c_p}$ described in the previous section.

\begin{lemma}\label{le:interp_hpc}
Let $1<p<\infty$ and $0<\theta<1$. 
Then, the following holds with equivalent norms
\begin{equation}\label{eq:interp_hpc}
(\h_1^c(\M), \h^c_p(\M))_{\theta}=\h^c_{q}(\M),
\end{equation}
where $\frac{1-\theta}{1}+\frac{\theta}{p}=\frac{1}{q}$.
\end{lemma}

\begin{proof} 
\textbf{Step 1:}
We first prove \eqref{eq:interp_hpc} in the case $1<q<p\leq 2$. 
As explained in Remark \ref{rk:density}, $\h^c_p(\M)$ can be identified with a subspace of \\
$L_p(\M\overline{\otimes}\mathrm{B}(\ell_2(\N^2)))$. 
Thus the interpolation between noncommutative $L_p$-spaces in \eqref{eq:cinterLp} 
gives the inclusion $(\h_1^c(\M), \h^c_p(\M))_{\theta}\subset \h^c_{q}(\M)$. 

The reverse inclusion needs more efforts. 
This can be shown using the equivalent quasinorm $N^c_p$ of $\|\cdot\|_{\h^c_p}$ defined previously. 
Let $x$ be an $L_2$-finite martingale such that $\|x\|_{\h^c_{q}}< 1$. 
By \eqref{eq:eq_norm} we have 
$$N^c_{q}(x)=\inf_W \Big[\tau\Big(\sum_nw_n^{1-2/q}|dx_{n+1}|^2\Big)\Big]^{1/2} 
< \Big(\frac{2}{q}\Big)^{1/2}.$$
Let $\{w_n\} \in W$ be such that 
\begin{equation}\label{eq:fix_wn}
\tau\Big(\displaystyle\sum_n w_n^{1-2/q}|dx_{n+1}|^2\Big)< \frac{2}{q}.
\end{equation} 
For $\varepsilon>0$ and $z\in S$ we define  
$$\begin{array}{ccl}
f_\varepsilon(z)&=&\exp(\varepsilon (z^2-\theta^2))
\displaystyle\sum_n dx_{n+1}w_n^{\frac{1}{2}-\frac{1}{q}}w_n^{\frac{1-z}{1}+\frac{z}{p}-\frac{1}{2}}\\
&=&\exp(\varepsilon (z^2-\theta^2))
\displaystyle\sum_n dx_{n+1}w_n^{1-(1-\frac{1}{p})z-\frac{1}{q}}.
\end{array}$$
Then $f_\varepsilon$ is continuous on $S$, analytic on $S_0$
and $f_\varepsilon(\theta)=x$. 
The term $\exp(\varepsilon (z^2-\theta^2))$ ensure that $f_\varepsilon(it)$ and $f_\varepsilon(1+it)$ 
tend to $0$ as $t$ goes to infinity. 
A direct computation gives for all $t\in \R$
$$ \tau\Big(\sum_nw_n^{-1}|d(f_\varepsilon)_{n+1}(it)|^2\Big)
=\exp(-2\varepsilon(t^2+\theta^2))\tau\Big(\sum_n w_n^{1-2/q}|dx_{n+1}|^2\Big).$$ 
By \eqref{eq:fix_wn} and \eqref{eq:eq_norm} we obtain
$$\|f_\varepsilon(it)\|_{\h^c_1}\leq \exp(\varepsilon) \Big(\frac{2}{q}\Big)^{1/2}.$$
Similarly,
$$\|f_\varepsilon(1+it)\|_{\h^c_p}\leq \exp(\varepsilon) \Big(\frac{2}{q}\Big)^{1/2}.$$
Thus $x=f_\varepsilon(\theta) \in (\h_1^c(\M), \h^c_p(\M))_{\theta}$ and 
$$\|x\|_{(\h_1^c(\M), \h^c_p(\M))_{\theta}}
\leq\exp(\varepsilon) \Big(\frac{2}{q}\Big)^{1/2};$$
whence
$$\|x\|_{(\h_1^c(\M), \h^c_p(\M))_{\theta}}\leq \Big(\frac{2}{q}\Big)^{1/2} \|x\|_{\h^c_{q}}.$$

\hspace{-0.4cm}\textbf{Step 2:} 
To obtain the general case, we use Wolff's interpolation theorem mentioned above. 
Let us first recall that for $1<v,s,q<\infty$ and $0<\eta<1$ 
such that $\frac{1}{q}=\frac{1-\eta}{v}+\frac{\eta}{s}$, 
we have with equivalent norms
\begin{equation}\label{eq:cinterp_hq}
(\h_v^c(\M), \h_s^c(\M))_{\eta}=\h^c_{q}(\M).
\end{equation}
Indeed, by Lemma $6.4$ of \cite{JX1}, 
$\h^c_p(\M)$ is one-complemented in $L_p^{\mathrm{cond}}(\M;\ell^c_2)$, for $1\leq p <\infty$. 
On the other hand, for $1<p<\infty$ 
the space $L_p^{\mathrm{cond}}(\M,\ell^c_2)$ is complemented in $L_p(\M,\ell^c_2(\N^2))$ via Stein's projection (Theorem $2.13$ of \cite{J}), 
and the column space $L_p(\M;\ell^c_2(\N^2))$ is a one-complemented subspace of $L_p(\M\overline{\otimes}\mathrm{B}(\ell_2(\N^2)))$. 
Thus, we conclude from \eqref{eq:cinterLp} that, by complementation, \eqref{eq:cinterp_hq} holds. 

We turn to the proof of \eqref{eq:interp_hpc}. 
Step $1$ shows that \eqref{eq:interp_hpc} holds in the case $1<p\leq 2$. 
Thus it remains to deal with the case $2<p<\infty$.  We divide the proof in two cases.\\
{\it Case 1:} $1<q<2<p<\infty$. 
Let $q<s<2$. Note that $1<q<s<p$, so there exist $0<\theta<1$ and $0<\phi<1$ such that 
$\frac{1-\theta}{1}+\frac{\theta}{s}=\frac{1}{q}$ and $\frac{1-\phi}{q}+\frac{\phi}{p}=\frac{1}{s}$. 
By \eqref{eq:cinterp_hq} we have
$$\h^c_s(\M)=(\h^c_q(\M),\h^c_p(\M))_\phi.$$
Furthermore, recall that $1<q<s<2$, so Step $1$ yields 
$$\h^c_q(\M)=(\h^c_1(\M),\h^c_s(\M))_\theta.$$
By Wolff's interpolation theorem \eqref{eq:wolff}, it follows that
$$\h^c_q(\M)=(\h^c_1(\M),\h^c_p(\M))_\varsigma,$$
where $\varsigma=\frac{\theta\phi}{1-\theta+\theta\phi}$. 
A simple computation shows that $\frac{1-\varsigma}{1}+\frac{\varsigma}{p}=\frac{1}{q}$.\\
{\it Case 2:} $2<q<p<\infty$. 
By a similar argument, we easily deduce this case from the previous one and \eqref{eq:cinterp_hq} using Wolff's theorem.

Note that in both cases, the density assumption of Wolff's theorem is ensured by Remark \ref{rk:density}.
\end{proof}

\begin{lemma}\label{le:interp_bmoc_hqc}
Let $1< q< p<\infty$.
Then, the following holds with equivalent norms
\begin{equation}\label{eq:interp_bmoc_hqc}
(\bmo^c(\M), \h^c_q(\M))_{\frac{q}{p}}=\h^c_{p}(\M).
\end{equation}
\end{lemma}

\begin{proof}
Applying the duality theorem $4.5.1$ of \cite{BL} to \eqref{eq:interp_hpc} 
we obtain \eqref{eq:interp_bmoc_hqc} in the case $1<q<p<\infty$ with $\theta=\frac{q}{p}$. 
Here we used the description of the dual space of $\h^c_{p}(\M)$ for $1\leq p <\infty$ mentioned in Remark \ref{rk:density}.
\end{proof}

\hspace{-0.4cm}{\it Proof of Theorem \ref{th:interp_bmoc_h1c}.}
We want to extend \eqref{eq:interp_bmoc_hqc} to the case $q=1$. 
To this aim we again use Wolff's interpolation theorem combined with the two previous lemmas. 
Let $1<q<p<\infty$. 
Then there exists $0<\phi<1$ such that $\frac{1-\phi}{1}+\frac{\phi}{p}=\frac{1}{q}$. 
We set $\theta=\frac{q}{p}$. 
Thus by Lemma \ref{le:interp_bmoc_hqc} we have
$$\h^c_p(\M)=(\bmo^c(\M),\h^c_q(\M))_\theta.$$
Moreover we deduce from Lemma \ref{le:interp_hpc} that
$$\h^c_q(\M)=(\h_1^c(\M),\h^c_p(\M))_\phi.$$
So Wolff's result yields
$$\h^c_p(\M)=(\bmo^c(\M),\h^c_1(\M))_\varsigma,$$
where $\varsigma=\frac{\theta\phi}{1-\theta+\theta\phi}$. 
An easy computation gives $\varsigma=\frac{1}{p}$, 
and this ends the proof of \eqref{eq:interpolation}
\cqd

\vspace{0.3cm}

The previous results concern the conditioned column Hardy space. 
We now consider the whole conditioned Hardy space, and get the analogue result.

\begin{theorem}\label{th:interp_bmo_h1}
Let $1<p<\infty$. 
Then, the following holds with equivalent norms
$$(\bmo(\M), \h_1(\M))_{\frac{1}{p}}=\h_{p}(\M).$$
\end{theorem}

The proof of Theorem \ref{th:interp_bmo_h1} is similar to that of Theorem \ref{th:interp_bmoc_h1c}. 
Indeed, we need the analogue of Lemma \ref{le:interp_hpc} for $\h_p(\M)$, 
and the result will follow from the same arguments. 
By Wolff's result, it thus remains to show that 
$(\h_1(\M),\h_p(\M))_{\theta}=\h_q(\M)$ for $1<p\leq 2$, 
where $\frac{1-\theta}{1}+\frac{\theta}{p}=\frac{1}{q}$. 
Recall that for $1\leq p\leq 2$ the space $\h_p(\M)$ is defined as a sum of three components
$$\h_p(\M)=\h_p^d(\M)+\h_p^c(\M)+\h_p^r(\M).$$
We will consider each component, and then will sum the interpolation results. 
The following lemma describe the behaviour of complex interpolation with addition.

\begin{lemma}\label{le:add_interpolation}
Let $(A_0,A_1)$ and $(B_0,B_1)$ be two compatible couples of Banach spaces. 
Then for $0<\theta<1$ we have 
$$(A_0,A_1)_\theta+(B_0,B_1)_\theta\subset (A_0+B_0,A_1+B_1)_\theta.$$
\end{lemma}

This result comes directly from the definition of complex interpolation. 

\begin{lemma}\label{le:interpolation_hpd}
Let $1\leq p_0<p_1\leq \infty, 0<\theta<1$. 
Then, the following holds with equivalent norms
$$(\h^d_{p_0}(\M),\h^d_{p_1}(\M))_\theta=\h^d_{p}(\M)$$
where $\frac{1}{p}=\frac{1-\theta}{p_0}+\frac{\theta}{p_1}$.
\end{lemma}

\begin{proof}
Recall that $\h^d_{p}(\M)$ consists of martingale difference sequences in $\ell_p(L_p(\M))$. 
So $\h^d_{p}(\M)$ is $2$-complemented in $\ell_p(L_p(\M))$ for $1\leq p\leq \infty$ via the projection
$$P : \left\{\begin{array}{ccc}
              \ell_p(L_p(\M)) & \longrightarrow & \h^d_{p}(\M)\\
		(a_n)_{n\geq 1} & \longmapsto & (\E_n(a_n)-\E_{n-1}(a_n))_{n\geq 1}
             \end{array}\right..$$
The fact that $\ell_p(L_p(\M))$ form an interpolation scale with respect to the complex interpolation yields the required result.
\end{proof}

\hspace{-0.4cm}{\it Proof of Theorem \ref{th:interp_bmo_h1}}
The row version of Lemma \ref{le:interp_hpc} holds true, as well, 
by considering the equivalent quasinorm $N^r_p$ of $\|\cdot\|_{\h^r_p}$. 
The diagonal version is ensured by Lemma \ref{le:interpolation_hpd}. 
Thus Lemma \ref{le:add_interpolation} yields the nontrivial inclusion 
$\h_{q}(\M)\subset (\h_1(\M), \h_p(\M))_{\theta}$ for $1<p\leq 2$. 
On the other hand, by \eqref{eq:burkh} we have $\h_p(\M)=L_p(\M)$ for $1<p<\infty$ 
and \eqref{eq:bmo_Linfty} yields by duality the inclusion $\h_1(\M)\subset L_1(\M)$. 
Hence \eqref{eq:cinterLp} gives the reverse inclusion 
$ (\h_1(\M), \h_p(\M))_{\theta}\subset\h_{q}(\M)$ for $1<p <\infty$. 
That establishs the analogue of Lemma \ref{le:interp_hpc} for $\h_p(\M)$, 
and Theorem \ref{th:interp_bmo_h1} follows using duality and Wolff's interpolation theorem. 
\cqd

\vspace{0.3cm}

We now consider the real method of interpolation. 
We show that the main result of this section remains true for this method. 
For $1<p<\infty$ and $1\leq r \leq \infty$, 
similarly to the construction of the space $L_{p}^{\cond}(\M;\ell_2^c)$ in Remark \ref{rk:density} we define 
the column and row subspaces of $L_{p,r}(\M\overline{\otimes} \mathrm{B}(\ell_2(\N^2)))$, 
denoted by $L_{p,r}^{\cond}(\M;\ell_2^c)$ and $L_{p,r}^{\cond}(\M;\ell_2^r)$, respectively.
Let $\h_{p,r}^c(\M)$ be the space of martingales $x$ such that $dx \in  L_{p,r}^{\cond}(\M;\ell_2^c)$. 

\begin{theorem}\label{th:rinterp_bmoc_h1c}
Let $1<p<\infty$ and $1\leq r \leq \infty$. 
Then, the following holds with equivalent norms
\begin{equation}\label{eq:real_interpolation}
(\bmo^c(\M), \h^c_1(\M))_{\frac{1}{p},r}=\h^c_{p,r}(\M).
\end{equation}
\end{theorem}

This result is a corollary of Theorem \ref{th:interp_bmoc_h1c}.

\begin{proof}
By a discussion similar to that at the beginning of Step $2$ 
in the proof of Lemma \ref{le:interp_hpc}, using  \eqref{eq:rinterLpq} we can show that 
for $1<v,s,q<\infty$, $1\leq r \leq \infty$ and $0<\eta<1$ such that $\frac{1}{q}=\frac{1-\eta}{v}+\frac{\eta}{s}$, 
we have with equivalent norms
\begin{equation}\label{eq:rinterp_hqr}
(\h_v^c(\M), \h_s^c(\M))_{\eta,r}=\h^c_{q,r}(\M).
\end{equation}
We deduce \eqref{eq:real_interpolation} from \eqref{eq:interpolation} 
using the reiteration theorem on real and complex interpolations. 
Let $1<p<\infty$.
Consider $1<p_0<p<p_1<\infty$. 
There exists $0<\eta<1$ such that
\begin{equation*}
\frac{1}{p}=\frac{1-\eta}{p_0}+\frac{\eta}{p_1}. 
\end{equation*}
By Theorem $4.7.2$ of \cite{BL} we obtain 
$$(\bmo^c(\M),\h^c_1(\M))_{\frac{1}{p},r}
=((\bmo^c(\M),\h^c_1(\M))_{\frac{1}{p_0}},(\bmo^c(\M),\h^c_1(\M))_{\frac{1}{p_1}})_{\eta,r}.$$
Then \eqref{eq:interpolation} yields
$$(\bmo^c(\M),\h^c_1(\M))_{\frac{1}{p},r}=(\h^c_{p_0}(\M),\h^c_{p_1}(\M))_{\eta,r}.$$
An application of \eqref{eq:rinterp_hqr} gives
\begin{equation*}
(\bmo^c(\M),\h^c_1(\M))_{\frac{1}{p},r}=\h^c_{p,r}(\M).
\end{equation*}
This ends the proof of \eqref{eq:real_interpolation}.
\end{proof}

\begin{rk}
Musat's result is a corollary of Theorem \ref{th:interp_bmoc_h1c}. 
By Davis' decomposition proved in \cite{P} we have 
$\H^c_p(\M)=\h^c_p(\M)+\h^d_p(\M)$ for $1\leq p<2$. 
So we can show the analogue of \eqref{eq:interp_hpc} for $1<p<2$ as follows, 
for $0<\theta<1$ and $\frac{1-\theta}{1}+\frac{\theta}{p}=\frac{1}{q}$
$$\begin{array}{cll}
&\H^c_{q}(\M)& \\
= &\h^c_{q}(\M)+\h^d_{q}(\M)&\\
=&(\h_1^c(\M), \h^c_p(\M))_{\theta}+(\h^d_1(\M), \h^d_p(\M))_{\theta} &
\quad\mbox{ by Lemmas \ref{le:interp_hpc} and \ref{le:interpolation_hpd}}\\
\subset&(\h_1^c(\M)+\h^d_1(\M), \h^c_p(\M)+\h^d_p(\M))_{\theta} &\quad\mbox{ by Lemma \ref{le:add_interpolation}}\\
=&(\H_1^c(\M), \H^c_p(\M))_{\theta}.&
\end{array}$$
On the other hand, recall that for $1\leq p <\infty$, $\H^c_p(\M)$ can be identified with 
the space of all $L_p$-martingales $x$ such that $dx\in L_p(\M;\ell^c_2)$. 
Thus we can consider $\H^c_p(\M)$ as a subspace of $L_p(\M\overline{\otimes}B(\ell_2))$ and the reverse inclusion follows. 
Then the same arguments, using duality and Wolff's theorem, yield Theorem $3.1$ of \cite{M}. 
Alternately, we can find Musat's result by defining an equivalent quasinorm for $\|\cdot\|_{\H^c_p(\M)}, 0< p \leq 2$ similar to $N_p^c$, 
as follows
$$\tilde{N}^c_p(x)=\inf_W \Big[\tau\Big(\sum_nw_n^{1-2/p}|dx_{n}|^2\Big)\Big]^{1/2}\approx\|x\|_{\H^c_p(\M)}.$$
Then all the previous proofs can be adapted to obtain the analogue results for $\H^c_p(\M)$.
\end{rk}

\section*{Appendix}
In Section $2$ we established the existence of an atomic decomposition for $\h_1(\M)$. 
The problem of explicitly constructing this decomposition remains open. 
One encounters some substantial difficulties in trying to adapt the classical atomic construction, which used stopping times, 
to the noncommutative setting. 
Note that explicit decompositions of martingales have already been constructed to establish weak-type inequalities (\cite{ran-weak,ran-cond}) 
and a noncommutative analogue of the Gundy's decomposition (\cite{par-ran-gu}). 
In these works, Cuculescu's projections played an important role 
and provide a good substitute for stopping times, 
which are a key tool for all these decompositions in the classical case. 
However, these projections do not seem to be powerful enough for the noncommutative atomic decomposition 
and for the noncommutative Davis' decomposition (see \cite{P}).

\begin{pb}\label{pb:explicit_decomp}
Find a constructive proof of Theorem \ref{th:c_atomic_decomp} or Theorem \ref{th:atomic_decomp}.
\end{pb}

\begin{pb}
 Construct an explicit Davis' decomposition \\
$\H_1(\M)=\h_1^c(\M)+\h_1^r(\M)+\h_1^d(\M)$. 
\end{pb}

It is also interesting to discuss the case of $\h_p$ for $0<p<1$. 
We define the noncommutative analogue of $(p,2)$-atoms as follows.

\begin{defi}
Let $0<p\leq1$. 
$a \in L_2 (\mathcal{M})$ is said to be a $(p,2)_c$-atom
with respect to $(\mathcal{M}_n)_{n \geq 1},$ if there exist $n \geq
1$ and a projection $e \in \mathcal{M}_n$ such that
\begin{enumerate}[{\rm (i)}]

\item  $\mathcal{E}_n (a) =0;$

\item  $r (a) \leq e;$

\item  $ \| a \|_2 \leq \tau (e)^{1/2-1/p}.$
\end{enumerate}
Replacing $\mathrm{(ii)}$ by $ \mathrm{(ii)'}~~l (a)\leq e,$ we get the notion of a $(p,2)_r$-atom.
\end{defi}
We define $\h_p^{c,\at}(\M)$ and  $\h_p^{r,\at}(\M)$ as in Definition \ref{df:h1at}. 
As for $p=1$, we have $\h_p^{c,\at}(\M) \subset \h_p^c(\M)$ contractively.

On the other hand, we can describe the dual space of $\h_p^{c,\at}(\M)$ as a Lipschitz space. 
For $\alpha\geq0$, we set 
$$\Lambda_\alpha^c(\M)=\big\{x\in L_2(\M):\|x\|_{\Lambda_\alpha^c}<\infty\big\}$$
with 
$$\|x\|_{\Lambda_\alpha^c}
=\sup_{n \geq 1} \sup_{e \in \mathcal{P}_n} \tau (e )^{-1/2-\alpha}\tau \big ( e |x - x_n |^2 \big)^{1/2}.$$
By a slight modification of the proof of Theorem \ref{th:duality_h1at} 
(by setting 
$y_e = \frac{(x- x_n )e}{ \| (x -x_n ) e \|_2 \tau (e)^{1/p-1/2}}$) 
we can show that $(\h_p^{c,\at}(\M))^*=\Lambda_\alpha^c(\M)$ for $0<p\leq1$, with $\alpha=1/p-1$.

At the time of this writing we do not know if $\h_p^{c,\at}(\M)$ coincides with $\h_p^c(\M)$. 
The problem of the atomic decomposition of $\h_p(\M)$ for $0<p<1$ is entirely open, 
and is related to Problem \ref{pb:explicit_decomp}.

\begin{pb}
 Does one have $\h_p^{c}(\M)=\h_p^{c,\at}(\M)$ for $0<p<1$?
\end{pb}

\begin{pb}
 Can we describe the dual space of $\h_p^c(\M)$ as a Lipschitz space for $0<p<1$ ?
\end{pb}

Another perspective of research concerns the interpolation results obtained in Section $4$. 
Recall that we define $\h_\infty^c(\M)$ (resp. $\h_\infty^r(\M)$) 
as the Banach space of the $L_{\infty}(\M)$-martingales $x$ such that 
$\sum_{k \geq 1} \E_{k-1}|dx_k|^2$ 
(respectively $\sum_{k \geq 1} \E_{k-1}|dx_k^*|^2$) converge for the weak operator topology. 
We set $\h_\infty(\M)=\h^c_\infty(\M)\cap \h^r_\infty(\M)\cap \h^d_\infty(\M)$. 
At the time of this writing we do not know if the interpolation result \eqref{eq:interpolation} remains true 
if we replace $\bmo(\M)$ by $\h_\infty(\M)$.

\begin{pb}
 Does one have $(\h^c_\infty(\M), \h^c_1(\M))_{\frac{1}{p}}=\h^c_{p}(\M)$ for $1< p<\infty$ ?
\end{pb}

\section*{Acknowledgment} The second named author is grateful to Professor Quanhua Xu for the
support of the two months visit to Laboratoire de Math\'{e}matiques,
Universit\'{e} de Franche-Comt\'{e} and the warm atmosphere at the
department, where a preliminary version of the paper was done.

\end{document}